\documentclass[11pt,a4paper]{article}
\usepackage{mathrsfs}
\usepackage{amsmath}
\usepackage{bbm}
\usepackage{amsmath,amsthm,amssymb,amscd}
\usepackage{latexsym}
\usepackage{hyperref}
\usepackage[numbers,sort&compress]{natbib}
\usepackage{hypernat}
\usepackage{geometry}
\usepackage{enumerate}
\usepackage[utf8]{inputenc}
\usepackage[english]{babel}
\usepackage{authblk}

\newcommand{\tr}{{\rm Tr\hskip -0.3em}~}

\newcommand{\normmm}[1]{{\left\vert\kern-0.25ex\left\vert\kern-0.25ex\left\vert #1
    \right\vert\kern-0.25ex\right\vert\kern-0.25ex\right\vert}}

\newtheorem{theorem}{Theorem}[section]

\newtheorem{corollary}[theorem]{Corollary}

\newtheorem{remark}[theorem]{Remark}

\theoremstyle{definition} \theoremstyle{remark}
\numberwithin{equation}{section}

\begin{document}
 \title{Variational Representations related to Quantum R\'{e}nyi Relative Entropies}
\author{Guanghua Shi\\
 School of Mathematical Sciences, Yangzhou University, Jiangsu, China
 \\shiguanghua@yzu.edu.cn}

\maketitle

\begin{abstract}
In this paper, we focus on variational representations of some matrix symmetric norm functions that are related to the quantum R\'{e}nyi relative entropy. Concretely, we obtain variational representations of the function $(A,B)\mapsto \normmm{(B^{q/2}K^*A^pKB^{q/2})^s}$ for symmetric norms by using the H\"{o}lder inequality and Young inequality. These variational expressions enable us to make the proofs of the convexity/concavity of the trace function $(A,B)\mapsto \tr (B^{q/2}K^*A^pKB^{q/2})^s$ more clear.
\end{abstract}

\section{Introduction}
Let $\mathcal{M}_n$ be the set of $n\times n$ matrices and $\mathcal{P}_n$ be the set of $n\times n$  positive definite matrices. A matrix $A\in \mathcal{P}_n$ with $\tr A=1$ is called a density matrix. Many of the statements in this paper are of special interest for density matrices but we will not make such restriction. For $A, B\in \mathcal{P}_n,$ the traditional R\'{e}nyi relative entropy (due to Petz \cite{Petz86}) is defined as
\begin{eqnarray}
D_{\alpha}(A||B)=\frac{1}{\alpha-1}\log \tr (A^{\alpha}B^{1-\alpha}), \quad \alpha\in (0,\infty)\backslash \{1\}.
\end{eqnarray}
A generalization of the traditional R\'{e}nyi relative entropy was introduced by M\"{u}ller-Lennert, Dupuis, Szehr, Fehr, Tomamichel \cite{MDSFT13} and Wilde, Winter, Yang \cite{WWY14}. This entropy is called the sandwiched R\'{e}nyi relative entropy, and is defined as
\begin{eqnarray}
\tilde{D}_{\alpha}(A||B)=\frac{1}{\alpha-1}\log  F_{\alpha}(A,B), \quad \alpha\in (0,\infty)\backslash \{1\},
\end{eqnarray}
where
\begin{eqnarray}
F_\alpha(A,B)= \tr \left(B^{\frac{1-\alpha}{2\alpha}}AB^{\frac{1-\alpha}{2\alpha}}\right)^{\alpha}, \quad \alpha\in (0,\infty).
\end{eqnarray}
The trace function $F_{\alpha}(A,B)$ is a parameterized version of the fidelity
\begin{eqnarray}
F(A,B)= \tr \left(B^{\frac{1}{2}}AB^{\frac{1}{2}}\right)^{\frac{1}{2}},
\end{eqnarray}
and is called the sandwiched quasi-relative entropy.
We should notice that (\cite{MDSFT13})
\begin{eqnarray}
\lim_{\alpha\rightarrow \infty} \tilde{D}_{\alpha}(A||B)=\|\log B^{-\frac{1}{2}}AB^{-\frac{1}{2}}\|,
\end{eqnarray}
where $\|\cdot \|$ is the operator norm. The expression in $(1.5)$ coincides with the Thompson metric
\begin{eqnarray*}
d_{T}(A,B)=\max\{\log \lambda_1(AB^{-1}),\log \lambda_1(BA^{-1})\}
\end{eqnarray*}
on $\mathcal{P}_n$ (see \cite{ACS00}), and is closely related to the max-relative entropy \begin{eqnarray*}D_{max}(A||B)=\log \lambda_1(AB^{-1})\end{eqnarray*}
in quantum information theory (see \cite{Dat09}).
Moreover,
\begin{eqnarray}
\lim_{\alpha\rightarrow 1} \tilde{D}_{\alpha}(A||B)=\frac{1}{\tr A}\tr A(\log A-\log B).
\end{eqnarray}
The expression in $(1.6)$ is the quantum relative entropy introduced by Umegaki \cite{Ume62}.
Recently, Audenaert and Datta \cite{AD15} unified the above R\'{e}nyi relative entropies and introduced the $\alpha-z$ R\'{e}nyi relative entropy
\begin{eqnarray}
D_{\alpha,z}(A||B)=\frac{1}{\alpha-1}\log \tr \left(B^{\frac{1-\alpha}{2z}}A^{\frac{\alpha}{z}}B^{\frac{1-\alpha}{2z}}\right)^{z}, \quad \alpha\in (0,\infty)\backslash \{1\},\quad z>0.
\end{eqnarray}
The $\alpha-z$ R\'{e}nyi relative entropy have appeared earlier in the paper \cite{JOPP12}.

The R\'{e}nyi relative entropies or more general quantum divergences $\mathfrak{D}(\cdot||\cdot)$ should satisfy the monotonicity under the quantum channel, i.e., the completely positive trace preserving map to make them have operational meaning. That is
\begin{eqnarray}
\mathfrak{D}(\Phi(A)||\Phi(B))\le \mathfrak{D}(A||B),
\end{eqnarray}
for all CPTP maps $\Phi$ and density matrices $A, B.$
This inequality is also known as the Data Processing Inequality. Essentially, the data processing inequality is equivalent to the joint convexity or concavity of the trace functions in the definition of the quantum divergence $\mathfrak{D}.$ And the key to solve the jointly convexity problems of such trace functions is to develop the related operator convexity theorems. This work was first started by Lieb \cite{Lieb73}, who gave the so called Lieb's concavity theorem and successfully solved the convexity of Wigner-Yanase-Dyson information. After that, Lieb's concavity theorem was the impetus to several related works. Such as Epstein's theorem \cite{Eps73}, Ando's convexity theorem \cite{Ando79} and their generalizations.

The convexity of the trace function $\tr A^{\alpha}B^{1-\alpha}$ in the traditional R\'{e}nyi relative entropy $D_{\alpha}$ can be established by Lieb's concavity theorem and Ando's convexity theorem.
The convexity of $F_\alpha(A,B)$ in the definition of the sandwiched R\'{e}nyi relative entropy was established by Frank and Lieb in \cite{FL13}, and is based on the extensions of Lieb's concavity theorem and Epstein's theorem. The trace function in the  $\alpha-z$ R\'{e}nyi relative entropy $D_{\alpha,z}$ can be abstracted into
\begin{eqnarray*}
\Psi_{p,q,s}(A,B)=\tr \left(B^{\frac{q}{2}}K^*A^pKB^{\frac{q}{2}}\right)^s.
\end{eqnarray*}
For which values of $p,q,s$ does $\Psi_{p,q,s}(A,B)$ satisfy convexity/concavity draw extensively attention in recent papers. We refer the readers to to \cite{CFL19} and also \cite{CFL16, CL08, Huang20, Z20} for the whole story of development of the convexity theorems related to $\Psi_{p,q,s}(A,B).$

In the development of Lieb's concavity theorem, there are several powerful methodologies. For instance, the complex analytic theory of Herglotz (Pick) functions used by Epstein \cite{Eps73} and developed by Hiai \cite{Hia13,Hia16}; the matrix tensor technique used by Ando \cite{Ando79}; the matrix perspective method introduced by Effros \cite{Eff09}; the complex interpolation technique used by Huang \cite{Huang20}; and certainly the variational method introduced by Carlen and Lieb \cite{CL08}.

In Carlen-Lieb's paper \cite{CL08}, they constructed variational formulas by using the tracial Young inequality and proved the trace function
\begin{eqnarray*}
\Upsilon_{p,q}(A)=\tr (B^*A^pB)^{\frac{q}{p}}
\end{eqnarray*}
is convex for $1\le p\le 2,$ and $q\ge 1,$ and is concave for $0\le p\le q\le 1.$
These results are extensions of the Lieb's theorem and the Epstein's theorem. In \cite{FL13}, by using variational method, Frank and Lieb proved that the trace function $F_{\alpha}(A,B)$ is jointly concave for $1/2\le \alpha <1$ and is jointly convex for $\alpha>1.$ In \cite{CFL16}, Carlen-Frank-Lieb considered the  convexity and concavity of the trace function $\Psi_{p,q,s}.$ They proved that when $1\le p\le 2, -1\le q<0,$  the function $\Psi_{p,q,s}$ is jointly convex for $s\ge \min \{1/(p-1),1/(1+q)\};$ when $p=2,$ $\Psi_{p,q,s}$ is jointly convex for $-1\le q<0$ and $s\ge 1/(p+q);$ and when $0\le p,q \le 1, 0\le s\le 1/(p+q),$ $\Psi_{p,q,s}$ is jointly concave. In \cite{Z20}, Zhang tackled with the  Audenaert-Datta conjecture \cite{AD15} and the Carlen-Frank-Lieb conjecture \cite{CFL19}. By using variational method, he proved that when $-1\le q<0, 1\le p \le 2, s\ge 1/(p+q),$ the trace function $\Psi_{p,q,s}$ is jointly convex. Hence together with other known results, the full range of $(p,q,s)$ for $\Psi_{p,q,s}$ to be joint convex/concave are given.

There are also other uses of the variational method in quantum information theory. The well-known Gibbs variational principle and the variational expressions established by Hiai-Petz \cite{HP13}, Tropp \cite{Tro11} and Shi-Hansen \cite{SH19} enable one to establish the relationship between quantum entropies and the trace functions related to exponential/logarithm functions.

In this paper, we focus on the variational method. We will give some different variational representations of $\Psi_{p,q,s}(A,B)$ by using the H\"{o}lder inequality, Young inequality and their reverse versions. Especially, we give the critical points of the variational representations in terms of operator geometric mean $A\sharp_{\alpha}B,$ which is defined as
\begin{eqnarray*}
A\sharp_{\alpha}B=A^{\frac{1}{2}}(A^{-\frac{1}{2}}BA^{-\frac{1}{2}})^{\alpha}A^{\frac{1}{2}}
\end{eqnarray*}
for $A,B \in \mathcal{P}_n$ and $\alpha\in \mathbb{R}.$
These representations will make the proof of the convexity/concavity of $\Psi_{p,q,s}(A,B)$ more clear and enable us to give some new extensions.

This paper is organized as follows. In section 2, we establish the reverse H\"{o}lder inequality and Young inequality for symmetric norms. In section 3, we derive variational representations of some matrix functionals related to quantum R\'{e}nyi relative entropies. Finally, in section 4, we recover the convexity/concavity of $\Psi_{p,q,s}(A,B)=\tr (B^{\frac{q}{2}}K^*A^pKB^{\frac{q}{2}})^s$ and give a new generalization of Lieb's concavity theorem in terms of symmetric anti-norms.

\section{Reverse H\"{o}lder Inequality and Young Inequality}

In this section, we consider the reverse H\"{o}lder inequality and Young Inequality for matrices.
Set $x=(x_1,\ldots, x_n)\in \mathbb{R}^n.$ A function $\Phi: \mathbb{R}^n\rightarrow \mathbb{R}_+$ is called a symmetric gauge function \cite{Bha97} if it satisfies the following conditions:
\begin{enumerate}
\item[(i)] $\Phi$ is a norm on $\mathbb{R}^n,$

\item[(ii)] $\Phi(Px)=\Phi(x)$ for all $x\in \mathbb{R}^n, \quad P\in S_n,$

\item[(iii)] $\Phi(\varepsilon_1 x_1,\ldots, \varepsilon_n x_n)=\Phi(x_1,\ldots,x_n)$ if $\varepsilon_j=\pm 1,$

\item[(iv)] $\Phi(1,0,\ldots,0)=1.$
\end{enumerate}

Symmetric gauge function is convex on $\mathbb{R}^n$ and is monotone on $\mathbb{R}^n_+.$ Hence if $x\prec_{w}y$ in $\mathbb{R}^n_+,$ we have
$\Phi(x)\le \Phi(y).$

Recall the notations: $|x|=(|x_1|,\ldots, |x_n|),$ and $|x|\le |y|$ if $|x_i|\le |y_i|$ for $1\le i\le n.$
From the scalar reverse Young inequality we have
\begin{eqnarray*}
\frac{1}{r}|x_i y_i|^r\ge \frac{1}{p}|x_i|^p+\frac{1}{q}|y_i|^q, \qquad \mbox{for} \quad r,p>0, q<0 \quad \mbox{and} \quad \frac{1}{r}=\frac{1}{p}+\frac{1}{q}.
\end{eqnarray*}
Then it follows that for $ r,p>0, q<0 $ and $\frac{1}{r}=\frac{1}{p}+\frac{1}{q},$
\begin{eqnarray}
\frac{1}{r}|x y|^r\ge \frac{1}{p}|x|^p+\frac{1}{q}|y|^q.
\end{eqnarray}

\begin{theorem} For $x, y\in  \mathbb{R}^n$ and symmetric gauge function $\Phi,$
\begin{eqnarray}
\Phi\left(|x y|^r\right)^{\frac{1}{r}}\ge \Phi\left(|x|^p\right)^{\frac{1}{p}}\Phi\left(|y|^q\right)^{\frac{1}{q}}
\end{eqnarray}
holds for $r,p>0, q<0$ with $\frac{1}{r}=\frac{1}{p}+\frac{1}{q}.$
\end{theorem}

\begin{proof}
By inequality $(2.1),$ it is easy to see
\begin{eqnarray*}
\frac{p}{r}|x y|^r-\frac{p}{q}|y|^q\ge |x|^p.
\end{eqnarray*}
Since $\frac{p}{r}, -\frac{p}{q}>0$ and $\frac{p}{r}+(-\frac{p}{q})=1,$ it follows from the monotonicity and convexity of $\Phi$ on $\mathbb{R}^n_+$ that
\begin{eqnarray*}
\frac{1}{r}\Phi\left(|x y|^r\right)\ge \frac{1}{p}\Phi\left(|x|^p\right)+\frac{1}{q}\Phi\left(|y|^q\right).
\end{eqnarray*}
For $t>0,$ by replacing $x, y$ by $tx$ and $t^{-1}y,$ we have
\begin{eqnarray*}
\Phi\left(|x y|^r\right)\ge \max_{t>0} \left\{\frac{r}{p}t^p\Phi\left(|x|^p\right)+\frac{r}{q} t^{-q}\Phi\left(|y|^q\right)\right\}.
\end{eqnarray*}
The function
\begin{eqnarray*}
\varphi (t)=\frac{r}{p}t^pa+\frac{r}{q} t^{-q}b, \qquad \mbox{where} \quad t, a, b>0,
\end{eqnarray*}
gets its maximum at the point
$t_0=(\frac{b}{a})^{\frac{1}{p+q}},$ and
\begin{eqnarray*}
\varphi(t_0)=a^{\frac{r}{p}}b^{\frac{r}{q}}.
\end{eqnarray*}
Hence we have
\begin{eqnarray*}
\Phi\left(|x y|^r\right)^{\frac{1}{r}}\ge \Phi\left(|x|^p\right)^{\frac{1}{p}}\Phi\left(|y|^q\right)^{\frac{1}{q}}.
\end{eqnarray*}
\end{proof}

Denote $s(A)$ the $n-$vector whose coordinates are the singular values of the matrix $A\in \mathcal{M}_n$ in the decreasing order, i.e. $s_1(A)\ge s_2(A)\ge \ldots\ge s_n(A).$ Given a symmetric gauge function $\Phi$ on $\mathbb{R}_n,$ the function
\begin{eqnarray*}
\normmm{A}:=\Phi(s(A))
\end{eqnarray*}
defines a symmetric norm (unitarily invariant norm) on $\mathcal{M}_n.$
It is well-known that the H\"{o}lder inequality for symmetric norm
\begin{eqnarray}
\normmm{|AB|^r}^{\frac{1}{r}}\le \normmm{|A|^p}^{\frac{1}{p}}\cdot \normmm{|B|^{q}}^{\frac{1}{q}}
\end{eqnarray}
holds for  $r, p, q>0$ with $\frac{1}{r}=\frac{1}{p}+\frac{1}{q},$ \cite[(IV.43)]{Bha97}.
Now we consider the reverse H\"{o}lder inequality for symmetric norms.

\begin{theorem} For symmetric norm  $\normmm{\cdot}$ and matrices $A, B$ with $B$ invertible, we have
\begin{eqnarray}
\normmm{|AB|^r}^{\frac{1}{r}}\ge \normmm{|A|^p}^{\frac{1}{p}}\cdot \normmm{|B|^{q}}^{\frac{1}{q}}
\end{eqnarray}
holds for $r, p>0, q<0$ and $\frac{1}{r}=\frac{1}{p}+\frac{1}{q}.$
\end{theorem}

\begin{proof}By Gelfand-Naimark Theorem (see \cite[Theorem 6.13]{HP14}) we have
\begin{eqnarray}
(s_i(A)s_{n-i+1}(B))\prec_{\log} s(AB).
\end{eqnarray}
Since $r>0,$ we have
\begin{eqnarray*}
(s^r_i(A)s_{n-i+1}^{r}(B))\prec_{\log} s^r(AB).
\end{eqnarray*}
It follows that
\begin{eqnarray*}
(s^r_i(A)s_{n-i+1}^{r}(B))\prec_{w} s^r(AB).
\end{eqnarray*}
Thus for matrices $A, B,$ with $B$ invertible and $r,p>0, q<0$ with $\frac{1}{r}=\frac{1}{p}+\frac{1}{q},$
\begin{eqnarray*}
\Phi \left(s^r(AB)\right)^{\frac{1}{r}}&\ge& \Phi \left(\left(s^r_i(A)s_{n-i+1}^{r}(B)\right)\right)^{\frac{1}{r}}\\&\ge&  \Phi\left(s^p(A)\right)^{\frac{1}{p}}\Phi\left((s_{n-i+1}^{q}(B))\right)^{\frac{1}{q}}
\\&=&\Phi\left(s(|A|^p)\right)^{\frac{1}{p}}\Phi\left(s(|B|^q)\right)^{\frac{1}{q}}.
\end{eqnarray*}
Hence it follows that
\begin{eqnarray*}
\normmm{|AB|^r}^{\frac{1}{r}}\ge \normmm{|A|^p}^{\frac{1}{p}}\cdot\normmm{|B|^{q}}^{\frac{1}{q}}.
\end{eqnarray*}
\end{proof}

The following corollaries are easy to get.

\begin{corollary} For matrices $A, B,$ with $B$ invertible, the reverse Young inequality for symmetric norm
\begin{eqnarray}
\frac{1}{r}\normmm{|AB|^r}\ge \frac{1}{p}\normmm{|A|^p}+ \frac{1}{q}\normmm{|B|^{q}}
\end{eqnarray}
holds for $r, p>0, q<0$ and $\frac{1}{r}=\frac{1}{p}+\frac{1}{q}.$
\end{corollary}

\begin{corollary} For matrices $A, B,$ with $B$ invertible,
the reverse tracial H\"{o}lder inequality
\begin{eqnarray*}
\left(\tr |AB|^r\right)^{\frac{1}{r}}\ge \left(\tr |A|^p\right)^{\frac{1}{p}}\left(\tr |B|^{q}\right)^{\frac{1}{q}}
\end{eqnarray*}
holds for $r, p>0, q<0$ and $\frac{1}{r}=\frac{1}{p}+\frac{1}{q}.$
\end{corollary}

\begin{corollary} For matrices $A, B,$ with $B$ invertible,
the reverse tracial Young inequality
\begin{eqnarray*}
\frac{1}{r}\tr |AB|^r\ge \frac{1}{p}\tr |A|^p+\frac{1}{q}\tr |B|^{q}
\end{eqnarray*}
holds for $r, p>0, q<0$ and $\frac{1}{r}=\frac{1}{p}+\frac{1}{q}.$
\end{corollary}

A symmetric anti-norm $\normmm{\cdot}_{!}$ on $\mathcal{P}_n$ is a non-negative continuous function that is positive homogeneous, unitary invariant and satisfying the super-additivity:
\begin{eqnarray*}\normmm{A+B}_!\ge \normmm{A}_!+\normmm{B}_! \quad \mbox{for all}\quad A,B\in \mathcal{P}_n.
\end{eqnarray*}
It is easy to see that a symmetric anti-norm on $\mathcal{P}_n$ is concave. And if $\normmm{\cdot}$ is a symmetric norm and $p>0,$ then for invertible matrix $A,$ $\normmm{A}_{!}:=\normmm{A^{-p}}^{-1/p}$ is a symmetric anti-norm.
For more information about the symmetric anti-norm and also reverse H\"{o}lder inequality for matrix we refer the readers to \cite{BH14}.

Now we define
\begin{eqnarray*}||A||_p=\left(\sum_{i=1}^n s_i^p(A)\right)^{1/p}
\end{eqnarray*}
for $p\in \mathbb{R}\setminus\{0\}.$ When $p<0,$ we set $A$ invertible, and in this case $||\cdot||_p$ is the negative Schatten anti-norm, which is a typical symmetric anti-norm. When $0<p<1,$ it is a quasi-norm and is also a symmetric anti-norm.  When $p\ge 1,$ it is the Schatten p-norm.

By Corollary 2.4, we can obtain the following well-known reverse H\"{o}lder inequality for Schatten norms.

\begin{corollary} Suppose $r, p>0$ and $q<0$ satisfying $1/r=1/p+1/q.$ Then for matrices $A,B$ with $B$ invertible,
\begin{eqnarray}
||AB||_r\ge ||A||_{p} ||B||_{q}.
\end{eqnarray}
\end{corollary}

The reverse H\"{o}lder inequality for Schatten norm $(2.7)$ can be found in many literatures. For example, we refer the readers to \cite{Bei13}, where Beigi showed that sandwiched R\'{e}nyi divergence satisfies the data processing inequality for $\alpha>1$ by using H\"{o}lder inequality and the theory of complex interpolation.

\section{Variational representations of some matrix functionals}

For symmetric norm $\normmm{\cdot},$ consider the function
\begin{eqnarray*}
\tilde{\Psi}_{p,q,s}(A,B)=\normmm{\left(B^{\frac{q}{2}}K^*A^pKB^{\frac{q}{2}}\right)^s}
\end{eqnarray*}
for $A, B\in \mathcal{P}_n, K\in \mathcal{M}_n$ and $p,q,s\in \mathbb{R}.$
We have the following variational representations:
\begin{theorem}
$(i)$ Let $s,p,q>0;$ or $s>0,$ $p,q<0.$ Then
\begin{eqnarray}
&&
\tilde{\Psi}_{p,q,s}(A,B)\nonumber\\
&=&\left\{
\begin{aligned}
&\min_{Z>0}\left\{\normmm{(Z^{-\frac{1}{2}}K^*A^pKZ^{-\frac{1}{2}})^{\frac{s(p+q)}{p}}}^{\frac{p}{p+q}}\cdot
\normmm{(Z^{\frac{1}{2}}B^qZ^{\frac{1}{2}})^{\frac{s(p+q)}{q}}}^{\frac{q}{p+q}}\right\}\\[1.5ex]
&\min_{Z>0}\left\{\frac{p}{p+q}\normmm{(Z^{-\frac{1}{2}}K^*A^pKZ^{-\frac{1}{2}})^{\frac{s(p+q)}{p}}}+
\frac{q}{p+q}\normmm{(Z^{\frac{1}{2}}B^qZ^{\frac{1}{2}})^{\frac{s(p+q)}{q}}}\right\}.
\end{aligned}
\right.
\end{eqnarray}

\noindent
$(ii)$ Let $s>0,p>0,q<0$ with $p+q>0;$ or $s>0,p<0,q>0$ with $p+q<0.$ Then
\begin{eqnarray}
&&
\tilde{\Psi}_{p,q,s}(A,B)\nonumber\\
&=&\left\{
\begin{aligned}
&\max_{Z>0}\left\{\normmm{(Z^{-\frac{1}{2}}K^*A^pKZ^{-\frac{1}{2}})^{\frac{s(p+q)}{p}}}^{\frac{p}{p+q}}\cdot
\normmm{ (Z^{\frac{1}{2}}B^qZ^{\frac{1}{2}})^{\frac{s(p+q)}{q}}}^{\frac{q}{p+q}}\right\}\\[1.5ex]
&\max_{Z>0}\left\{\frac{p}{p+q}\normmm{ (Z^{-\frac{1}{2}}K^*A^pKZ^{-\frac{1}{2}})^{\frac{s(p+q)}{p}}}+
\frac{q}{p+q}\normmm{(Z^{\frac{1}{2}}B^qZ^{\frac{1}{2}})^{\frac{s(p+q)}{q}}}\right\}.
\end{aligned}
\right.
\end{eqnarray}
\end{theorem}

\begin{proof}
From H\"{o}lder inequality and Young inequality and their reverse versions, we have for $S, T\in \mathcal{M}_n,$
\begin{eqnarray}
\normmm{|ST|^{r_0}} &\le& \normmm{|S|^{r_1}}^{\frac{r_0}{r_1}}\normmm{|T|^{r_2}}^{\frac{r_0}{r_2}}
\\&\le& \frac{r_0}{r_1}\normmm{|S|^{r_1}}+ \frac{r_0}{r_2}\normmm{|T|^{r_2}} \qquad (r_0,r_1,r_2>0, \frac{1}{r_0}=\frac{1}{r_1}+\frac{1}{r_2});
\end{eqnarray}
and for $S, T\in \mathcal{M}_n$ with $T$ invertible,
\begin{eqnarray}
\normmm{|ST|^{r_0}}&\ge& \normmm{|S|^{r_1}}^{\frac{r_0}{r_1}} \normmm{|T|^{r_2}}^{\frac{r_0}{r_2}}
\\&\ge& \frac{r_0}{r_1}\normmm{|S|^{r_1}}+ \frac{r_0}{r_2}\normmm{|T|^{r_2}} \qquad (r_0,r_1>0,r_2<0, \frac{1}{r_0}=\frac{1}{r_1}+\frac{1}{r_2}).
\end{eqnarray}

Under the conditions of $(i),$ set $$r_0=2s,\quad r_1=\frac{2s(p+q)}{p},\quad r_2=\frac{2s(p+q)}{q},$$ then the conditions $r_0,r_1,r_2>0$ and $\frac{1}{r_0}=\frac{1}{r_1}+\frac{1}{r_2}$ hold. Then by setting $S=A^{\frac{p}{2}}KZ^{-\frac{1}{2}}$ and $T=Z^{\frac{1}{2}}B^{\frac{q}{2}}$ in inequality $(3.3)$ and $(3.4)$, we have
\begin{eqnarray*}
\tilde{\Psi}_{p,q,s}(A,B)
&=&\normmm{(B^{\frac{q}{2}}K^*A^pKB^{\frac{q}{2}})^s}\\
&=&\normmm{(B^{\frac{q}{2}}Z^{\frac{1}{2}}Z^{-\frac{1}{2}}K^*A^{\frac{p}{2}}A^{\frac{p}{2}}KZ^{-\frac{1}{2}}Z^{\frac{1}{2}}B^{\frac{q}{2}})^s}\\
&=&\normmm{|A^{\frac{p}{2}}KZ^{-\frac{1}{2}}\cdot Z^{\frac{1}{2}}B^{\frac{q}{2}}|^{r_0}}\\
&\le&\normmm{|A^{\frac{p}{2}}KZ^{-\frac{1}{2}}|^{r_1}}^{\frac{r_0}{r_1}} \cdot \normmm{|Z^{\frac{1}{2}}B^{\frac{q}{2}}|^{r_2}}^{\frac{r_0}{r_2}}\\
&=& \normmm{(Z^{-\frac{1}{2}}K^*A^pKZ^{-\frac{1}{2}})^{\frac{s(p+q)}{p}}}^{\frac{p}{p+q}}\cdot
\normmm{(Z^{\frac{1}{2}}B^qZ^{\frac{1}{2}})^{\frac{s(p+q)}{q}}}^{\frac{q}{p+q}}\\
&\le& \frac{p}{p+q}\normmm{(Z^{-\frac{1}{2}}K^*A^pKZ^{-\frac{1}{2}})^{\frac{s(p+q)}{p}}}+
\frac{q}{p+q}\normmm{(Z^{\frac{1}{2}}B^qZ^{\frac{1}{2}})^{\frac{s(p+q)}{q}}}.
\end{eqnarray*}
When
\begin{eqnarray*}
Z=B^{-q}\sharp_{\frac{q}{p+q}}(K^*A^pK),
\end{eqnarray*}
we have
\begin{eqnarray*}
\normmm{ (Z^{-\frac{1}{2}}K^*A^pKZ^{-\frac{1}{2}})^{\frac{s(p+q)}{p}}}
&=&\normmm{ ((K^*A^pK)^{\frac{1}{2}}Z^{-1}(K^*A^pK)^{\frac{1}{2}})^{\frac{s(p+q)}{p}}}\\
&=&\normmm{ ((K^*A^pK)^{\frac{1}{2}}((K^*A^pK)^{-1}\sharp_{\frac{p}{p+q}}B^q)(K^*A^pK)^{\frac{1}{2}})^{\frac{s(p+q)}{p}}}\\
&=&\normmm{ ((K^*A^pK)^{\frac{1}{2}}B^q(K^*A^pK)^{\frac{1}{2}})^{\frac{p}{p+q}\frac{s(p+q)}{p}}}\\
&=&\normmm{(B^{\frac{q}{2}}K^*A^pKB^{\frac{q}{2}})^s};
\end{eqnarray*}
and
\begin{eqnarray*}
\normmm{(Z^{\frac{1}{2}}B^qZ^{\frac{1}{2}})^{\frac{s(p+q)}{q}}}
&=&\normmm{ (B^{\frac{q}{2}}ZB^{\frac{q}{2}})^{\frac{s(p+q)}{q}}}\\
&=&\normmm{ (B^{\frac{q}{2}}K^*A^pKB^{\frac{q}{2}})^{\frac{q}{p+q}\frac{s(p+q)}{q}}}\\
&=&\normmm{ (B^{\frac{q}{2}}K^*A^pKB^{\frac{q}{2}})^s}.
\end{eqnarray*}
Then it follows that
\begin{eqnarray*}
&&
\tilde{\Psi}_{p,q,s}(A,B)=\normmm{(B^{\frac{q}{2}}K^*A^pKB^{\frac{q}{2}})^s}\\
&=&\left\{
\begin{aligned}
&\min_{Z>0}\left\{\normmm{(Z^{-\frac{1}{2}}K^*A^pKZ^{-\frac{1}{2}})^{\frac{s(p+q)}{p}}}^{\frac{p}{p+q}}\cdot
\normmm{(Z^{\frac{1}{2}}B^qZ^{\frac{1}{2}})^{\frac{s(p+q)}{q}}}^{\frac{q}{p+q}}\right\}\\[1.5ex]
&\min_{Z>0}\left\{\frac{p}{p+q}\normmm{ (Z^{-\frac{1}{2}}K^*A^pKZ^{-\frac{1}{2}})^{\frac{s(p+q)}{p}}}+
\frac{q}{p+q}\normmm{(Z^{\frac{1}{2}}B^qZ^{\frac{1}{2}})^{\frac{s(p+q)}{q}}}\right\}.
\end{aligned}
\right.
\end{eqnarray*}

Under the conditions of $(ii),$ set $$r_0=2s,\quad  r_1=\frac{2s(p+q)}{p},\quad r_2=\frac{2s(p+q)}{q}.$$  Then we have $r_0>0,r_1>0,r_2<0$ and $\frac{1}{r_0}=\frac{1}{r_1}+\frac{1}{r_2}.$ Now set $S=A^{\frac{p}{2}}KZ^{-\frac{1}{2}}$ and $T=Z^{\frac{1}{2}}B^{\frac{q}{2}}$ in inequalities $(3.5)$ and $(3.6).$ Following a similar argument as above, we can obtain
\begin{eqnarray*}
&&
\tilde{\Psi}_{p,q,s}(A,B)=\normmm{(B^{\frac{q}{2}}K^*A^pKB^{\frac{q}{2}})^s}\\
&=&\left\{
\begin{aligned}
&\max_{Z>0}\left\{\normmm{ (Z^{-\frac{1}{2}}K^*A^pKZ^{-\frac{1}{2}})^{\frac{s(p+q)}{p}}}^{\frac{p}{p+q}}\cdot
\normmm{ (Z^{\frac{1}{2}}B^qZ^{\frac{1}{2}})^{\frac{s(p+q)}{q}}}^{\frac{q}{p+q}}\right\}\\[1.5ex]
&\max_{Z>0}\left\{\frac{p}{p+q}\normmm{(Z^{-\frac{1}{2}}K^*A^pKZ^{-\frac{1}{2}})^{\frac{s(p+q)}{p}}}+
\frac{q}{p+q}\normmm{(Z^{\frac{1}{2}}B^qZ^{\frac{1}{2}})^{\frac{s(p+q)}{q}}}\right\}.
\end{aligned}
\right.
\end{eqnarray*}
\end{proof}

\begin{theorem} $(i)$ Let $s,q>0$ with $s<1/q.$  Then
\begin{eqnarray}
&&
\tilde{\Psi}_{p,q,s}(A,B)\nonumber\\
&=&\left\{
\begin{aligned}
&\min_{Z>0}\left\{\normmm{(Z^{-\frac{1}{2}}K^*A^pKZ^{-\frac{1}{2}})^{\frac{s}{1-sq}}}^{1-sq}\cdot
\normmm{ (Z^{\frac{1}{2}}B^qZ^{\frac{1}{2}})^{\frac{1}{q}}}^{sq}\right\}\\[1.5ex]
&\min_{Z>0}\left\{(1-sq)\normmm{(Z^{-\frac{1}{2}}K^*A^pKZ^{-\frac{1}{2}})^{\frac{s}{1-sq}}}+
sq \normmm{(Z^{\frac{1}{2}}B^qZ^{\frac{1}{2}})^{\frac{1}{q}}}\right\}.
\end{aligned}
\right.
\end{eqnarray}

\noindent
$(ii)$ Let $s>0,q<0.$ Then
\begin{eqnarray}
&&
\tilde{\Psi}_{p,q,s}(A,B)\nonumber\\
&=&\left\{
\begin{aligned}
&\max_{Z>0}\left\{\normmm{ (Z^{-\frac{1}{2}}K^*A^pKZ^{-\frac{1}{2}})^{\frac{s}{1-sq}}}^{1-sq}\cdot
\normmm{(Z^{\frac{1}{2}}B^qZ^{\frac{1}{2}})^{\frac{1}{q}}}^{sq}\right\}\\[1.5ex]
&\max_{Z>0}\left\{(1-sq)\normmm{(Z^{-\frac{1}{2}}K^*A^pKZ^{-\frac{1}{2}})^{\frac{s}{1-sq}}}+
sq \normmm{ (Z^{\frac{1}{2}}B^qZ^{\frac{1}{2}})^{\frac{1}{q}}}\right\}.
\end{aligned}
\right.
\end{eqnarray}

\end{theorem}

\begin{proof} Under the conditions of $(i),$ set $r_0=2s, r_1=\frac{2s}{1-sq}, r_2=\frac{2}{q},$ then $r_0,r_1,r_2>0$ and $\frac{1}{r_0}=\frac{1}{r_1}+\frac{1}{r_2}$ hold. Now
set $S=A^{\frac{p}{2}}KZ^{-\frac{1}{2}}$ and $T=Z^{\frac{1}{2}}B^{\frac{q}{2}}$ in inequality $(3.3)$ and $(3.4)$, then we have
\begin{eqnarray*}
\tilde{\Psi}_{p,q,s}(A,B)
&=&\normmm{(B^{\frac{q}{2}}K^*A^pKB^{\frac{q}{2}})^s}\\
&=&\normmm{(B^{\frac{q}{2}}Z^{\frac{1}{2}}Z^{-\frac{1}{2}}K^*A^{\frac{p}{2}}A^{\frac{p}{2}}KZ^{-\frac{1}{2}}Z^{\frac{1}{2}}B^{\frac{q}{2}})^s}\\
&\le&\normmm{\;|A^{\frac{p}{2}}KZ^{-\frac{1}{2}}|^{r_1}}^{\frac{r_0}{r_1}} \cdot \normmm{\;|Z^{\frac{1}{2}}B^{\frac{q}{2}}|^{r_2}}^{\frac{r_0}{r_2}}\\
&=& \normmm{(Z^{-\frac{1}{2}}K^*A^pKZ^{-\frac{1}{2}})^{\frac{s}{1-sq}}}^{1-sq}\cdot
\normmm{(Z^{\frac{1}{2}}B^qZ^{\frac{1}{2}})^{\frac{1}{q}}}^{sq}\\
&\le& (1-sq)\normmm{(Z^{-\frac{1}{2}}K^*A^pKZ^{-\frac{1}{2}})^{\frac{s}{1-sq}}}+
sq \normmm{(Z^{\frac{1}{2}}B^qZ^{\frac{1}{2}})^{\frac{1}{q}}}.
\end{eqnarray*}
When
\begin{eqnarray*}
Z=B^{-q}\sharp_{sq}(K^*A^pK),
\end{eqnarray*}
we have
\begin{eqnarray*}
\normmm{(Z^{-\frac{1}{2}}K^*A^pKZ^{-\frac{1}{2}})^{\frac{s}{1-sq}}}
=\normmm{(B^{\frac{q}{2}}K^*A^pKB^{\frac{q}{2}})^s};
\end{eqnarray*}
and
\begin{eqnarray*}
\normmm{(Z^{\frac{1}{2}}B^qZ^{\frac{1}{2}})^{\frac{1}{q}}}
=\normmm{(B^{\frac{q}{2}}K^*A^pKB^{\frac{q}{2}})^s}.
\end{eqnarray*}
Hence it follows that
\begin{eqnarray*}
&&
\tilde{\Psi}_{p,q,s}(A,B)=\normmm{\left(B^{\frac{q}{2}}K^*A^pKB^{\frac{q}{2}}\right)^s}\\
&=&\left\{
\begin{aligned}
&\min_{Z>0}\left\{\normmm{ (Z^{-\frac{1}{2}}K^*A^pKZ^{-\frac{1}{2}})^{\frac{s}{1-sq}}}^{1-sq}\cdot
\normmm{(Z^{\frac{1}{2}}B^qZ^{\frac{1}{2}})^{\frac{1}{q}}}^{sq}\right\}\\[1.5ex]
&\min_{Z>0}\left\{(1-sq)\normmm{(Z^{-\frac{1}{2}}K^*A^pKZ^{-\frac{1}{2}})^{\frac{s}{1-sq}}}+
sq \normmm{(Z^{\frac{1}{2}}B^qZ^{\frac{1}{2}})^{\frac{1}{q}}}\right\}.
\end{aligned}
\right.
\end{eqnarray*}

Similarly, under the conditions of $(ii)$ and by using inequalities $(3.5)$ and $(3.6),$ we can get
\begin{eqnarray*}
&&
\tilde{\Psi}_{p,q,s}(A,B)=\normmm{\left(B^{\frac{q}{2}}K^*A^pKB^{\frac{q}{2}}\right)^s}\\
&=&\left\{
\begin{aligned}
&\max_{Z>0}\left\{\normmm{ (Z^{-\frac{1}{2}}K^*A^pKZ^{-\frac{1}{2}})^{\frac{s}{1-sq}}}^{1-sq}\cdot
\normmm{(Z^{\frac{1}{2}}B^qZ^{\frac{1}{2}})^{\frac{1}{q}}}^{sq}\right\}\\[1.5ex]
&\max_{Z>0}\left\{(1-sq)\normmm{(Z^{-\frac{1}{2}}K^*A^pKZ^{-\frac{1}{2}})^{\frac{s}{1-sq}}}+
sq\normmm{(Z^{\frac{1}{2}}B^qZ^{\frac{1}{2}})^{\frac{1}{q}}}\right\}.
\end{aligned}
\right.
\end{eqnarray*}
\end{proof}

Consider the trace function
\begin{eqnarray*}
\Psi_{p,q,s}(A,B)=\tr \left(B^{\frac{q}{2}}K^*A^pKB^{\frac{q}{2}}\right)^s,
\end{eqnarray*}
for $A, B\in \mathcal{P}_n, K\in \mathcal{M}_n$ and $p,q,s\in \mathbb{R}.$
We have the following variational representations:
\begin{corollary}
$(i)$ Let $s,p,q>0;$ or $s>0,$ $p,q<0.$ Then
\begin{eqnarray}
&&
\Psi_{p,q,s}(A,B)\nonumber\\
&=&\left\{
\begin{aligned}
&\min_{Z>0}\left\{\left(\tr (Z^{-\frac{1}{2}}K^*A^pKZ^{-\frac{1}{2}})^{\frac{s(p+q)}{p}}\right)^{\frac{p}{p+q}}\cdot
\left(\tr (Z^{\frac{1}{2}}B^qZ^{\frac{1}{2}})^{\frac{s(p+q)}{q}}\right)^{\frac{q}{p+q}}\right\}\\[1.5ex]
&\min_{Z>0}\left\{\frac{p}{p+q}\tr (Z^{-\frac{1}{2}}K^*A^pKZ^{-\frac{1}{2}})^{\frac{s(p+q)}{p}}+
\frac{q}{p+q}\tr (Z^{\frac{1}{2}}B^qZ^{\frac{1}{2}})^{\frac{s(p+q)}{q}}\right\}.
\end{aligned}
\right.
\end{eqnarray}

\noindent
$(ii)$ Let $s>0,p>0,q<0$ with $p+q>0;$ or $s>0,p<0,q>0$ with $p+q<0.$ Then
\begin{eqnarray}
&&
\Psi_{p,q,s}(A,B)\nonumber\\
&=&\left\{
\begin{aligned}
&\max_{Z>0}\left\{\left(\tr (Z^{-\frac{1}{2}}K^*A^pKZ^{-\frac{1}{2}})^{\frac{s(p+q)}{p}}\right)^{\frac{p}{p+q}}\cdot
\left(\tr (Z^{\frac{1}{2}}B^qZ^{\frac{1}{2}})^{\frac{s(p+q)}{q}}\right)^{\frac{q}{p+q}}\right\}\\[1.5ex]
&\max_{Z>0}\left\{\frac{p}{p+q}\tr (Z^{-\frac{1}{2}}K^*A^pKZ^{-\frac{1}{2}})^{\frac{s(p+q)}{p}}+
\frac{q}{p+q}\tr (Z^{\frac{1}{2}}B^qZ^{\frac{1}{2}})^{\frac{s(p+q)}{q}}\right\}.
\end{aligned}
\right.
\end{eqnarray}
\end{corollary}

\begin{corollary} $(i)$ Let $s,q>0$ with $s<1/q.$  Then
\begin{eqnarray}
&&
\Psi_{p,q,s}(A,B)\nonumber\\
&=&\left\{
\begin{aligned}
&\min_{Z>0}\left\{\left(\tr (Z^{-\frac{1}{2}}K^*A^pKZ^{-\frac{1}{2}})^{\frac{s}{1-sq}}\right)^{1-sq}\cdot
\left(\tr (Z^{\frac{1}{2}}B^qZ^{\frac{1}{2}})^{\frac{1}{q}}\right)^{sq}\right\}\\[1.5ex]
&\min_{Z>0}\left\{(1-sq)\tr (Z^{-\frac{1}{2}}K^*A^pKZ^{-\frac{1}{2}})^{\frac{s}{1-sq}}+
sq \tr (Z^{\frac{1}{2}}B^qZ^{\frac{1}{2}})^{\frac{1}{q}}\right\}.
\end{aligned}
\right.
\end{eqnarray}

\noindent
$(ii)$ Let $s>0,q<0.$ Then
\begin{eqnarray}
&&
\Psi_{p,q,s}(A,B)\nonumber\\
&=&\left\{
\begin{aligned}
&\max_{Z>0}\left\{\left(\tr (Z^{-\frac{1}{2}}K^*A^pKZ^{-\frac{1}{2}})^{\frac{s}{1-sq}}\right)^{1-sq}\cdot
\left(\tr (Z^{\frac{1}{2}}B^qZ^{\frac{1}{2}})^{\frac{1}{q}}\right)^{sq}\right\}\\[1.5ex]
&\max_{Z>0}\left\{(1-sq)\tr (Z^{-\frac{1}{2}}K^*A^pKZ^{-\frac{1}{2}})^{\frac{s}{1-sq}}+
sq \tr (Z^{\frac{1}{2}}B^qZ^{\frac{1}{2}})^{\frac{1}{q}}\right\}.
\end{aligned}
\right.
\end{eqnarray}

\end{corollary}

Set $s=t, p=1, q=\frac{1-t}{t}$ and $K=I$ in Corollary 3.3 or Corollary 3.4, we can obtain:
\begin{corollary}
Let $A,B\in \mathcal{P}_n.$ If $0\le t\le 1,$ then
\begin{eqnarray*}
&&F_t(A,B)=\tr \left(B^{\frac{1-t}{2t}}AB^{\frac{1-t}{2t}}\right)^t\\
&=&\left\{
\begin{aligned}
&\min_{Z\in \mathcal{P}_n}\left\{\left(\tr Z^{-\frac{1}{2}}AZ^{-\frac{1}{2}}\right)^t\cdot \left(\tr (Z^{\frac{1}{2}}B^{\frac{1-t}{t}}Z^{\frac{1}{2}})^{\frac{t}{1-t}}\right)^{1-t}\right\}\\[1.5ex]
&\min_{Z\in \mathcal{P}_n}\left\{t\tr Z^{-\frac{1}{2}}AZ^{-\frac{1}{2}}+
(1-t) \tr (Z^{\frac{1}{2}}B^{\frac{1-t}{t}}Z^{\frac{1}{2}})^{\frac{t}{1-t}}\right\}.
\end{aligned}
\right.
\end{eqnarray*}
If $t\ge 1,$ then
\begin{eqnarray*}
&&F_t(A,B):=\tr \left(B^{\frac{1-t}{2t}}AB^{\frac{1-t}{2t}}\right)^t\\
&=&\left\{
\begin{aligned}
&\max_{Z\in \mathcal{P}_n}\left\{\left(\tr Z^{-\frac{1}{2}}AZ^{-\frac{1}{2}}\right)^t\cdot \left(\tr (Z^{\frac{1}{2}}B^{\frac{1-t}{t}}Z^{\frac{1}{2}})^{\frac{t}{1-t}}\right)^{1-t}\right\}\\[1.5ex]
&\max_{Z\in \mathcal{P}_n}\left\{t\tr Z^{-\frac{1}{2}}AZ^{-\frac{1}{2}}+
(1-t) \tr (Z^{\frac{1}{2}}B^{\frac{1-t}{t}}Z^{\frac{1}{2}})^{\frac{t}{1-t}}\right\}.
\end{aligned}
\right.
\end{eqnarray*}
\end{corollary}

The variational expressions of $F_t(A,B)$ for $t\in (0,1)$ was obtained in \cite{FL13}. See also \cite{BFT17} and \cite{BJL18}.

\section{Generalizing Lieb's Concavity Theorem via variational method}

Now we consider the convexity or concavity of $\tilde{\Psi}_{p,q,s}(A,B)$ and $\Psi_{p,q,s}(A,B)$ by using the variational method. Before doing so, we recall some convexity/concavity theorems about the matrix function
\begin{eqnarray*}
\Upsilon_{p,s}(A)=\tr (K^*A^pK)^s.
\end{eqnarray*}
Here we only consider the case of $s>0$.
\begin{theorem} Fix $K\in \mathcal{M}_n.$ Then for $A\in \mathcal{P}_n$ we have
\begin{enumerate}
\item[(i)]  If  $ 0\le p\le 1 $ and $ 0<s\le 1/p, $  then $ \Upsilon_{p,s} $ is concave.
\item[(ii)] If  $ -1\le p\le 0 $ and $ s>0,  $ then $ \Upsilon_{p,s} $ is convex.
\item[(iii)] If  $ 1\le p\le 2 $  and $ s\ge 1/p, $  then $ \Upsilon_{p,s} $ is convex.
\end{enumerate}
\end{theorem}

We firstly recover the following well-known conclusions by using Corollary 3.4. For more information of Theorem 4.2 we refer the readers to \cite{CFL19, Z20}.
\begin{theorem}
Let $K\in \mathcal{M}_n$ be arbitrary, and $A,B\in \mathcal{P}_n.$
\begin{enumerate}
\item[(i)]  If  $ 0\le p,q \le1 $ and $ 0<s\le 1/(p+q), $  then $ \Psi_{p,q,s}(A,B) $ is jointly concave.
\item[(ii)] If  $ -1\le p,q\le 0 $ and $ s>0,  $ then $ \Psi_{p,q,s}(A,B) $ is jointly convex.
\item[(iii)] If  $ 1\le p\le 2, -1\le q\le 0, (p,q)\neq (1,-1),$  and $ s\ge 1/(p+q), $  then $ \Psi_{p,q,s}(A,B) $ is jointly convex.
\end{enumerate}
\end{theorem}

\begin{proof}
Under the conditions of $(i),$ we have
\begin{eqnarray*}
0\le p\le 1, \quad \frac{s}{1-sq}=\frac{1}{s^{-1}-q}\le \frac{1}{p},\qquad \mbox{and} \qquad
0\le q\le 1.
\end{eqnarray*}
Hence it follows from Theorem 4.1 $(i)$ that
\begin{eqnarray*}
\tr (Z^{-\frac{1}{2}}K^*A^pKZ^{-\frac{1}{2}})^{\frac{s}{1-sq}}
\end{eqnarray*}
is concave in $A$.
And
\begin{eqnarray*}
\tr (Z^{\frac{1}{2}}B^qZ^{\frac{1}{2}})^{\frac{1}{q}}
\end{eqnarray*}
is concave in $B$.
Since $(1-sq)$ and $sq$ are both positive, we have
\begin{eqnarray*}
(1-sq)\tr (Z^{-\frac{1}{2}}K^*A^pKZ^{-\frac{1}{2}})^{\frac{s}{1-sq}}+
sq \tr (Z^{\frac{1}{2}}B^qZ^{\frac{1}{2}})^{\frac{1}{q}}
\end{eqnarray*}
is concave in $(A,B)$. Then by the variational representation $(3.11)$ and Lemma 13 of \cite{CFL19}  we have $\Psi_{p,q,s}(A,B)$ is jointly concave.

Under the conditions of $(ii),$ we have
\begin{eqnarray*}
-1\le p\le 0, \quad \frac{s}{1-sq}=\frac{1}{s^{-1}-q}>0,\qquad \mbox{and} \qquad
-1\le q\le 0.
\end{eqnarray*}
Hence it follows from Theorem 4.1 $(ii)$ that
\begin{eqnarray*}
\tr (Z^{-\frac{1}{2}}K^*A^pKZ^{-\frac{1}{2}})^{\frac{s}{1-sq}}
\end{eqnarray*}
is convex in $A$.
And it follows from Theorem 4.1 $(i)$ that
\begin{eqnarray*}
\tr (Z^{\frac{1}{2}}B^qZ^{\frac{1}{2}})^{\frac{1}{q}}
\end{eqnarray*}
is concave in $B$.
Since $(1-sq)$ is positive and $sq$ is negative, we have
\begin{eqnarray*}
(1-sq)\tr (Z^{-\frac{1}{2}}K^*A^pKZ^{-\frac{1}{2}})^{\frac{s}{1-sq}}+
sq \tr (Z^{\frac{1}{2}}B^qZ^{\frac{1}{2}})^{\frac{1}{q}}
\end{eqnarray*}
is convex in $(A,B)$.
Hence by the variational representation $(3.12)$  and Lemma 13 of \cite{CFL19}, we have $\Psi_{p,q,s}(A,B)$ is jointly convex.

Under the conditions of $(iii),$ we have
\begin{eqnarray*}
1\le p\le 2, \quad \frac{s}{1-sq}=\frac{1}{s^{-1}-q}\ge \frac{1}{p},\qquad \mbox{and} \qquad
-1\le q\le 0.
\end{eqnarray*}
Hence it follows from Theorem 4.1 $(iii)$ that
\begin{eqnarray*}
\tr (Z^{-\frac{1}{2}}K^*A^pKZ^{-\frac{1}{2}})^{\frac{s}{1-sq}}
\end{eqnarray*}
is convex in $A$.
And it follows from Theorem 4.1 $(i)$ that
\begin{eqnarray*}
\tr (Z^{\frac{1}{2}}B^qZ^{\frac{1}{2}})^{\frac{1}{q}}
\end{eqnarray*}
is concave in $B$.
Since $(1-sq)$ is positive and $sq$ is negative, we have
\begin{eqnarray*}
(1-sq)\tr (Z^{-\frac{1}{2}}K^*A^pKZ^{-\frac{1}{2}})^{\frac{s}{1-sq}}+
sq \tr (Z^{\frac{1}{2}}B^qZ^{\frac{1}{2}})^{\frac{1}{q}}
\end{eqnarray*}
is convex in $(A,B)$.
Hence by the variational representation $(3.12)$ and Lemma 13 of \cite{CFL19}, we have $\Psi_{p,q,s}(A,B)$ is jointly convex.
\end{proof}

More generally, Hiai \cite{Hia13} proved the following results, which can be viewed as generalizations of the Epstein's theorem for symmetric (anti-)norms.
\begin{theorem} Set $K\in \mathcal{M}_n.$ And set $\normmm{\cdot}_{!}$ be symmetric anti-norm and $\normmm{\cdot}$ be symmetric norm for matrix.
\begin{enumerate}
\item[(i)]  If  $ 0< p\le 1 $ and $ 0<s\le 1/p, $  then $ \normmm{(K^*A^pK)^s}_{!} $ is concave for $A\in \mathcal{P}_n$.
\item[(ii)] If  $ -1\le p< 0 $ and $ s>0,  $ then $ \normmm{(K^*A^pK)^s} $ is convex for $A\in \mathcal{P}_n$.
\item[(iii)] If  $ 1\le p\le 2 $  and $ s\ge 1/p, $  then $ \normmm{(K^*A^pK)^s} $ is convex for $A\in \mathcal{P}_n$.
\end{enumerate}
\end{theorem}

We now consider the extension of the Lieb's concavity theorem for symmetric anti-norm.

\begin{theorem}
Let $K\in \mathcal{M}_n$ and $A,B\in \mathcal{P}_n.$ Let $\normmm{\cdot}_{!}$ be a symmetric anti-norm which ensures the H\"{o}lder inequality:
\begin{eqnarray}
\normmm{\;|ST|^r\;}_!^{\frac{1}{r}}\le \normmm{\;|S|^{r_1}\;}_!^{\frac{1}{r_1}}\cdot \normmm{\;|T|^{r_2}\;}_!^{\frac{1}{r_2}},
\end{eqnarray}
for $S,T\in P_n,$ and $r, r_1, r_2>0$ with $\frac{1}{r}=\frac{1}{r_1}+\frac{1}{r_2}.$
Then for $ 0\le p,q \le1 $ and $ 0<s\le 1/(p+q), $  we have
\begin{eqnarray*} (A,B)\mapsto \normmm{(B^{\frac{q}{2}}K^*A^pKB^{\frac{q}{2}})^s}_{!}
\end{eqnarray*}
 is jointly concave.
\end{theorem}

\begin{proof} Since
\begin{eqnarray*}
0\le p \le 1, \quad \frac{s(p+q)}{p}\le \frac{1}{p},\qquad \mbox{and} \qquad
0\le q\le 1, \quad
\frac{s(p+q)}{q}\le \frac{1}{q},
\end{eqnarray*}
it follows from Theorem 4.3 $(i)$ that
\begin{eqnarray*}
\normmm{ (Z^{-\frac{1}{2}}K^*A^pKZ^{-\frac{1}{2}})^{\frac{s(p+q)}{p}}}_{!}
\end{eqnarray*}
is concave in $A$, and
\begin{eqnarray*}
\normmm{ (Z^{\frac{1}{2}}B^qZ^{\frac{1}{2}})^{\frac{s(p+q)}{q}}}_{!}
\end{eqnarray*}
is concave in $B$.
The H\"{o}lder inequality $(4.1)$ ensures the corresponding Young inequality and hence a similar version of variational representation for symmetric anti-norm:
\begin{eqnarray*}
&&\normmm{(B^{\frac{q}{2}}K^*A^pKB^{\frac{q}{2}})^s}_{!} \\
&=& \min_{Z\in \mathcal{P}_n}\left\{
\frac{p}{p+q}\normmm{ (Z^{-\frac{1}{2}}K^*A^pKZ^{-\frac{1}{2}})^{\frac{s(p+q)}{p}}}_!+
\frac{q}{p+q}\normmm{ (Z^{\frac{1}{2}}B^qZ^{\frac{1}{2}})^{\frac{s(p+q)}{q}}}_!\right\}
\end{eqnarray*}
Hence,
it follows that $ \normmm{(B^{\frac{q}{2}}K^*A^pKB^{\frac{q}{2}})^s}_{!} $ is jointly concave in $(A,B)$.
\end{proof}

\begin{remark}
The Schatten quasi-norm $\|\cdot\|_p$ for $p\in (0,1)$ is a symmetric anti-norm, which is concave and satisfying the H\"{o}lder inequality and also the reverse H\"{o}lder inequality.
\end{remark}

{\bf Acknowledgements.}
The author acknowledges support from the Natural Science Foundation of Jiangsu Province for Youth, Grant No: BK20190874.

\end{document}